\documentclass[11pt,reqno]{amsart}

\usepackage{xspace}
\usepackage{amsmath,amsthm,amsfonts,amssymb,latexsym,mathrsfs,color,extarrows}
\usepackage{hyperref}

\newtheorem{theorem}{Theorem}

\newtheorem{proposition}[theorem]{Proposition}

\newtheorem{definition}[theorem]{Definition}
\newtheorem{example}[theorem]{Example}

\newcommand{\el}{{\rm el\,}}
\newcommand{\ol}{{\rm ol\,}}
\newcommand{\m}{{\rm M}}

\newcommand{\fap}{{\rm fap\,}}
\newcommand{\dasc}{{\rm dasc\,}}
\newcommand{\desp}{{\rm dp\,}}
\newcommand{\Dasc}{{\rm Dasc\,}}
\newcommand{\Desp}{{\rm DP\,}}
\newcommand{\Lap}{{\rm LAP\,}}
\newcommand{\NDP}{{\rm NDP\,}}
\newcommand{\fdes}{{\rm fdes\,}}
\newcommand{\fasc}{{\rm fasc\,}}

\newcommand{\plat}{{\rm plat\,}}
\newcommand{\ap}{{\rm ap\,}}
\newcommand{\lap}{{\rm lap\,}}

\newcommand{\des}{{\rm des\,}}

\newcommand{\mmn}{\mathcal{M}_{2n}}

\newcommand{\msn}{\mathfrak{S}_n}

\newcommand{\ms}{\mathfrak{S}}
\newcommand{\maj}{{\rm maj\,}}
\newcommand{\inv}{{\rm inv\,}}

\newcommand{\mq}{\mathcal{Q}}
\newcommand{\mqn}{\mathcal{Q}_n}

\newcommand{\asc}{{\rm asc\,}}

\newcommand{\Eulerian}[2]{\genfrac{<}{>}{0pt}{}{#1}{#2}}
\newcommand{\Stirling}[2]{\genfrac{\{}{\}}{0pt}{}{#1}{#2}}

\title{The ascent-plateau statistics on Stirling permutations}
\author[S.-M.~Ma]{Shi-Mei Ma}
\address{School of Mathematics and Statistics,
        Northeastern University at Qinhuangdao,
         Hebei 066004, P.R. China}
\email{shimeimapapers@163.com (S.-M. Ma)}
\author[J.~Ma]{Jun Ma}
\address{Department of mathematics, Shanghai jiao tong university, Shanghai, P.R. China}
\email{majun904@sjtu.edu.cn(J.~Ma)}
\author[Y.-N. Yeh]{Yeong-Nan Yeh}
\address{Institute of Mathematics,
        Academia Sinica, Taipei, Taiwan}
\email{mayeh@math.sinica.edu.tw (Y.-N. Yeh)}
\begin{document}

\maketitle
\begin{abstract}
In this paper, several variants of the ascent-plateau statistic are introduced, including flag ascent-plateau, double ascent and descent-plateau.   We first study the flag ascent-plateau statistic on Stirling permutations by using context-free grammars.
We then present a unified refinement of the ascent polynomials and the ascent-plateau polynomials. In particular, by using Foata and Strehl's group action, we prove two bistatistics over the set of Stirling permutations of order $n$ are equidistributed.

\bigskip

\end{abstract}
\date{\today}
\section{Introduction}
A {\it Stirling permutation} of order $n$ is a permutation of the multiset $\{1,1,2,2,\ldots,n,n\}$ such that
for each $i$, $1\leq i\leq n$, all entries between the two occurrences of $i$ are larger than $i$.
Denote by $\mqn$ the set of {\it Stirling permutations} of order $n$.
Let $\sigma=\sigma_1\sigma_2\cdots\sigma_{2n}\in\mqn$. For $1\leq i\leq 2n$, we say that an index $i$ is a {\it descent} of $\sigma$ if
$\sigma_{i}>\sigma_{i+1}$ or $i=2n$, and we say that an index $i$ is an {\it ascent} of $\sigma$ if $\sigma_{i}<\sigma_{i+1}$ or $i=1$.
Hence the index $i=1$ is always an ascent and $i=2n$ is always a descent.
Moreover, a {\it plateau} of $\sigma$ is an index $i$ such that $\sigma_{i}=\sigma_{i+1}$, where $1\leq i\leq 2n-1$.
Let $\des(\sigma),\asc(\sigma)$ and $\plat(\sigma)$ be the numbers of descents, ascents and plateaus of $\sigma$, respectively.

Stirling permutations were defined by Gessel and Stanley~\cite{Gessel78}, and they proved that
$$(1-x)^{2k+1}\sum_{n=0}^\infty \Stirling{n+k}{n}x^n=\sum_{\sigma\in\mq_k}x^{\des{\sigma}},$$
where $\Stirling{n}{k}$ is the {\it Stirling number of the second kind}, i.e.,
the number of ways to partition a set of $n$ objects into $k$ non-empty subsets.
A classical result of B\'ona~\cite{Bona08} says that descents, ascents and plateaus have the same distribution over $\mqn$, i.e.,
$$\sum_{\sigma\in\mqn}x^{\des{\sigma}}=\sum_{\sigma\in\mqn}x^{\asc{\sigma}}=\sum_{\sigma\in\mqn}x^{\plat{\sigma}}.$$
This equidistributed result and associated multivariate polynomials have been extensively studied by Janson, Kuba, Panholzer, Haglund, Chen et al., see~\cite{Chen17,Haglund12,Janson08,Janson11} and references therein.

Recently, Ma and Toufik~\cite{Ma15} introduced the definition of
ascent-plateau statistic and presented a combinatorial interpretation of the $1/k$-Eulerian polynomials.
The purpose of this paper is to explore variants of the ascent-plateau statistic. In the following, we collect some definitions, notation and results that will be needed throughout this paper.

\begin{definition}
An occurrence of an {\it ascent-plateau} of $\sigma\in\mqn$ is an index $i$ such that $\sigma_{i-1}<\sigma_{i}=\sigma_{i+1}$, where $i\in\{2,3,\ldots,2n-1\}$. An occurrence of a {\it left ascent-plateau} is an index $i$ such that $\sigma_{i-1}<\sigma_{i}=\sigma_{i+1}$, where $i\in\{1,2,\ldots,2n-1\}$ and $\sigma_0=0$.
\end{definition}
Let $\ap(\sigma)$ and $\lap(\sigma)$ be the numbers of ascent-plateaus and left ascent-plateaus of $\sigma$, respectively.
For example, $\ap(442\textbf{3}32115\textbf{6}65)=2$ and $\lap(\textbf{4}42\textbf{3}32115\textbf{6}65)=3$.

Define $$M_n(x)=\sum_{\sigma\in \mqn}x^{\ap(\sigma)},~N_n(x)=\sum_{\sigma\in \mqn}x^{\lap(\sigma)}.$$
According to~\cite[Theorem 2,~Theorem 3]{Ma15}, we have
\begin{equation}\label{Mxt}
M(x,t)=\sum_{n\geq 0}M_n(x)\frac{t^n}{n!}=\sqrt{\frac{x-1}{x-e^{2t(x-1)}}},
\end{equation}
\begin{equation}\label{Nxt}
N(x,t=)\sum_{n\geq 0}N_n(x)\frac{t^n}{n!}=\sqrt{\frac{1-x}{1-xe^{2t(1-x)}}}.
\end{equation}
It should be noted that the polynomials $M_n(x)$ and $N_n(x)$ are also enumerative polynomials of perfect matchings.
A {\it perfect matching} of $[2n]$ is a partition of $[2n]$ into $n$ blocks of size $2$. Let $\mmn$ be the set of perfect matchings of $[2n]$.
Let $\el(\m)$ (resp. $\ol(\m)$) be the number of blocks of $\m\in\mmn$ with even (resp. odd) larger entries.
According to~\cite{MaYeh16}, we have $$M_n(x)=\sum_{\m\in\mmn}x^{\ol(\m)},~N_n(x)=\sum_{\m\in\mmn}x^{\el(\m)}.$$

Let $\#C$ denote the cardinality of a set $C$.
Let $\msn$ denote the symmetric group of all permutations $\pi=\pi(1)\pi(2)\dots \pi(n)$ of $[n]$,
where $[n]=\{1,2,\ldots,n\}$. A {\it descent} of
$\pi$ is an index $i\in[n-1]$ such that $\pi(i)>\pi(i+1)$.
For $\pi\in\msn$, let $\des(\pi)$ be the number of descents of $\pi$.
The classical Eulerian polynomials are defined by
$$A_n(x) =\sum_{\pi\in\msn}x^{\des(\pi)}.$$
The {\it hyperoctahedral group} $B_n$ is the group of signed permutations of the set $\pm[n]$ such that $\pi(-i)=-\pi(i)$ for all $i$, where $\pm[n]=\{\pm1,\pm2,\ldots,\pm n\}$. Throughout this paper, we always identify a signed permutation $\pi=\pi(1)\cdots\pi(n)$ with the word $\pi(0)\pi(1)\cdots\pi(n)$, where $\pi(0)=0$.
For each $\pi\in B_n$, we define
\begin{equation*}
\begin{split}
\des_A(\pi)&=\#\{i\in [n-1]: \pi(i)>\pi(i+1)\},\\
\des_B(\pi)&=\#\{i\in \{0,1,2\ldots,n-1\}: \pi(i)>\pi(i+1)\}.
\end{split}
\end{equation*}
It is clear that $$\sum_{\pi\in B_n}x^{\des_A(\pi)}=2^nA_n(x).$$
Following~\cite{Adin2001}, the {\it flag descents} of $\pi\in B_n$ is defined by
$$\fdes(\pi)=\left\{
               \begin{array}{ll}
                 2\des_A(\pi)+1, & \hbox{if $\pi(1)<0$;} \\
                 2\des_A(\pi), & \hbox{otherwise.}
               \end{array}
             \right.
$$
The Eulerian polynomial of type $B$ and the flag descent polynomial are respectively defined by
\begin{align*}
  B_n(x) =\sum_{\pi\in B_n}x^{\des_B(\pi)},~
  F_n(x) =\sum_{\pi\in B_n}x^{\fdes(\pi)}.
\end{align*}

Very recently, we studied the following combinatorial expansions (see~\cite[Section 4]{Ma17}):
\begin{equation*}\label{NnxAnx}
2^nxA_n(x)=\sum_{k=0}^n\binom{n}{k}N_k(x)N_{n-k}(x),~
B_n(x)=\sum_{k=0}^n\binom{n}{k}N_k(x)M_{n-k}(x).
\end{equation*}
It is now well known that $F_n(x)=(1+x)^nA_n(x)$ (see~\cite[Theorem 4.4]{Adin2001}).
This paper is motivated by exploring an expansion of $F_n(x)$ in terms of some enumerative polynomials of Stirling permutations.

This paper is organized as follows.
In Section~\ref{Section2}, we present a combinatorial expansion of $F_n(x)$.
In Section~\ref{Section03}, we study a multivariate enumerative polynomials of Stirling permutations.
In particular, we consider Foata and Strehl's group action on Stirling permutations.
\section{The flag descent polynomials and flag ascent-plateau polynomials}\label{Section2}

Context-free grammar is a powerful tool to study exponential structures (see~\cite{Chen17,Ma17} for instance).
In this section, we first present a grammatical description of the flag descent polynomials by using grammatical labeling introduced by Chen and Fu~\cite{Chen17}. And then, we study the flag ascent-plateau statistics over Stirling permutations.

\subsection{Context-free grammars}\label{Section02}
\hspace*{\parindent}

For an alphabet $A$, let $\mathbb{Q}[[A]]$ be the rational commutative ring of formal power
series in monomials formed from letters in $A$. Following~\cite{Chen93}, a context-free grammar over
$A$ is a function $G: A\rightarrow \mathbb{Q}[[A]]$ that replaces a letter in $A$ by a formal function over $A$.
The formal derivative $D$ is a linear operator defined with respect to a context-free grammar $G$. More precisely,
the derivative $D=D_G$: $\mathbb{Q}[[A]]\rightarrow \mathbb{Q}[[A]]$ is defined as follows:
for $x\in A$, we have $D(x)=G(x)$; for a monomial $u$ in $\mathbb{Q}[[A]]$, $D(u)$ is defined so that $D$ is a derivation,
and for a general element $q\in\mathbb{Q}[[A]]$, $D(q)$ is defined by linearity.

Let us now recall a result on
context-free grammars.
\begin{proposition}[{\cite[Theorem~10]{Ma131}}]\label{Ma13}
Let $A=\{x,y,z\}$ and
\begin{equation}\label{grammar-flag}
G=\{x\rightarrow xyz, y\rightarrow yz^2, z\rightarrow y^2z\}.
\end{equation}
For $n\geq 0$, we have
\begin{equation}\label{Dnxy}
 D^n(xy)=xy\sum_{\pi\in B_n}y^{\fdes(\pi)}z^{2n-\fdes(\pi)}.
\end{equation}
Moreover,
\begin{align*}
 D^n(y^2)& =y^2\sum_{\pi\in B_n}y^{2\des_A(\pi)}z^{2n-2\des_A(\pi)}, \\
 D^n(yz)& =yz\sum_{\pi\in B_n}y^{2\des_B(\pi)}z^{2n-2\des_B(\pi)},\\
 D^n(y)& =y\sum_{\pi\in \mqn}y^{2\ap(\sigma)}z^{2n-2\ap(\sigma)},\\
 D^n(z)& =z\sum_{\pi\in \mqn}y^{2\lap(\sigma)}z^{2n-2\lap(\sigma)}.
\end{align*}
\end{proposition}


The grammatical labeling is illustrated in the following proof of~\eqref{Dnxy}.
Let $\pi\in B_n$. As usual, denote by $\overline{i}$ the negative element $-i$.
We define an {\it ascent} (resp. a descent) of $\pi$ to be a position $i\in \{0,1,2\ldots,n-1\}$ such that
$\pi(i)<\pi(i+1)$ (resp. $\pi(i)>\pi(i+1)$).
Now we give a labeling of $\pi\in B_n$ as follows:
\begin{itemize}
  \item [\rm ($L_1$)]If $i\in[n-1]$ is an ascent, then put a superscript label $z$ and a subscript label $z$ right after $\pi(i)$;
 \item [\rm ($L_2$)]If $i\in[n-1]$ is a descent, then put a superscript label $y$ and a subscript label $y$ right after $\pi(i)$;
\item [\rm ($L_3$)]If $\pi(1)>0$, then put a superscript label $z$ and a subscript label $x$ right after $\pi(0)$;
\item [\rm ($L_4$)]If $\pi(1)<0$, then put a superscript label $x$ and a subscript label $y$ right after $\pi(0)$;
\item [\rm ($L_5$)] Put a superscript label $y$ and a subscript label $z$ at the end of $\pi$.
\end{itemize}
Note that the weight of $\pi$ is given by $w(\pi)=xy^{\fdes(\pi)+1}z^{\fasc(\pi)+1}$.

Let $$F_n(i,j)=\{\pi\in B_n: \fdes(\pi)=i, \fasc(\pi)=j\}.$$
When $n=1$, we have $F_1(0,1)=\{0^z_x1^y_z\}$ and $F_1(1,0)=\{0^x_{y}\overline{1}^y_z\}$.
Note that
$D(xy)=xyz^2+xy^2z$. Thus the sum of weights of the elements of $B_1$ is given by $D(xy)$.

Suppose we get all labeled permutations in $F_n(i,j)$ for all $i,j,k$, where $n\geq 1$. Let
$\pi'\in B_{n+1}$ be obtained from $\pi\in F_n(i,j)$ by inserting the entry $n+1$ or $\overline{n+1}$.
We distinguish the following five cases:
\begin{itemize}
  \item [\rm ($c_1$)] Let $i\in [n-1]$ be an ascent. If we insert $n+1$ (resp.~$\overline{n+1}$) right after $\pi(i)$, then $\pi'\in F_{n+1}(i+2,j)$, and the insertion of $n+1$ (resp.~$\overline{n+1}$) corresponds to applying the rule $z\rightarrow y^2z$ to the superscript (resp. subscript) label $z$ associated with $\pi(i)$.
\item [\rm ($c_2$)] Let $i\in [n-1]$ be a descent. If we insert $n+1$ (resp.~$\overline{n+1}$) right after $\pi(i)$, then $\pi'\in F_{n+1}(i,j+2)$, and the insertion of $n+1$ (resp.~$\overline{n+1}$) corresponds to applying the rule $y\rightarrow yz^2$ to the superscript (resp. subscript) label $y$ associated with $\pi(i)$.
    \item [\rm ($c_3$)] If we insert $n+1$ (resp. $\overline{n+1}$) at the end of $\pi$, then $\pi'\in F_{n+1}(i,j+2)$ (resp. $\pi'\in F_{n+1}(i+2,j)$), and the insertion of $n+1$ (resp. $\overline{n+1}$) corresponds to applying the rule $y\rightarrow yz^2$ (resp. $z\rightarrow y^2z$) to the label $y$ (resp. $z$) at the end of $\pi$;
    \item [\rm ($c_4$)] If $\pi(1)>0$ and we insert $n+1$ (resp. $\overline{n+1}$) immediately before $\pi(1)$, then $\pi'\in F_{n+1}(i+2,j)$ (resp. $\pi'\in F_{n+1}(i+1,j+1)$), and the insertion of $n+1$ (resp. $\overline{n+1}$ corresponds to applying the rule $z\rightarrow y^2z$ (resp. $x\rightarrow xyz$) to the label $z$ (resp. $x$) right after $\pi(0)$;
  \item [\rm ($c_5$)] If $\pi(1)<0$ and we insert $n+1$ (resp. $\overline{n+1}$) immediately before $\pi(1)$, then $\pi'\in F_{n+1}(i+1,j+1)$ (resp. $\pi'\in F_{n+1}(i,j+2)$), and the insertion of $n+1$ (resp. $\overline{n+1}$ corresponds to applying the rule $x\rightarrow xyz$ (resp. $y\rightarrow yz^2$) to the label $x$ (resp. $y$) right after $\pi(0)$.
\end{itemize}

In general, the insertion of $n+1$ (resp. $\overline{n+1}$) into $\pi$ corresponds
to the action of the formal derivative $D$ on a superscript label (resp. subscript label).
By induction, we get a grammatical proof of~\eqref{Dnxy}.

\begin{example}
For example, let $\pi=04\overline{3}152$. Then $\pi$ can be generated as follows:
\begin{align*}
0^z_x1^\textbf{y}_z&\mapsto 0^z_x1^z_z2^y_z;\\
0^z_\textbf{x}1^z_z2^y_z&\mapsto 0^x_y\overline{3}^z_z1^z_z2^y_z;\\
 0^\textbf{x}_y\overline{3}^z_z1^z_z2^y_z&\mapsto  0^z_x4^y_y\overline{3}^z_z1^z_z2^y_z;\\
0^z_x4^y_y\overline{3}^z_z1^\textbf{z}_z2^y_z&\mapsto 0^z_x4^y_y\overline{3}^z_z1^z_z5^y_y2^y_z.
\end{align*}
\end{example}

\subsection{The flag ascent-plateau statistic}

\begin{definition}
Let $\sigma=\sigma_1\sigma_2\cdots\sigma_{2n}\in\mqn$.
The number of flag ascent-plateau of $\sigma$ is defined by $$\fap(\sigma)=\left\{
               \begin{array}{ll}
                 2\ap(\sigma)+1, & \hbox{if $\sigma_1=\sigma_2$;} \\
                 2\ap(\sigma), & \hbox{otherwise.}
               \end{array}
             \right.
$$
\end{definition}

We can now present the first main result of this paper.
\begin{theorem}
Let $D$ be the formal derivative with respect to the grammar~\eqref{grammar-flag}. For $n\geq 1$, we have
\begin{equation}\label{Dnx}
D^n(x)=x\sum_{\sigma\in\mqn}y^{\fap(\sigma)}z^{2n-\fap(\sigma)}.
\end{equation}
Therefore,
\begin{equation}\label{Fnx}
\sum_{\pi\in B_n}x^{\fdes(\pi)}=\sum_{k=0}^n\binom{n}{k}\sum_{\sigma\in\mq_{k}}x^{\fap(\sigma)}\sum_{\sigma\in\mq_{n-k}}x^{2\ap(\sigma)}.
\end{equation}
\end{theorem}
\begin{proof}
We first introduce a grammatical labeling of $\sigma\in \mqn$ as follows:
\begin{itemize}
  \item [\rm ($L_1$)]If $i\in \{2,3,\ldots,2n-1\}$ is an ascent-plateau, then put a superscript label $y$ immediately before $\sigma_i$ and a superscript label $y$ right after $\sigma_i$;
 \item [\rm ($L_2$)]If $\sigma_1=\sigma_2$, then put a superscript label $y$ immediately before $\sigma_1$ and a superscript $x$ right after $\sigma_1$;
\item [\rm ($L_3$)]If $\sigma_1<\sigma_2$, then put a superscript label $x$ immediately before $\sigma_1$;
\item [\rm ($L_4$)]The rest of positions in $\sigma$ are labeled by a superscript label $z$.
\end{itemize}
Note that the weight of $\sigma$ is given by $$w(\sigma)=xy^{\fap(\sigma)}z^{2n-\fap(\sigma)}.$$
For example, The labeling of $1223314554$ and $661223314554$ are respectively given as follows:
$$^x1^y2^y2^y3^y3^z1^z4^y5^y5^z4^z,~^y6^x6^z1^y2^y2^y3^y3^z1^z4^y5^y5^z4^z.$$

We then prove~\eqref{Dnx} by induction. Let $S_n(i)=\{\sigma\in\mqn: \fap(\sigma)=i\}$.
For $n=1$, we have $S_1(1)=\{^y1^x1^z\}$.
For $n=2$, the elements of $\mq_2$ are labeled as follows:
$$S_2(1)=\{^y2^x2^z1^z1^z\},~S_2(2)=\{^x1^y2^y2^z1^z\},~S_2(3)=\{^y1^x1^y2^y2^z\}.$$
Note that $D(x)=xyz$ and $D^2(x)=xyz(y^2+yz+z^2)$. Hence the result holds for $n=1,2$.
Suppose we get all labeled Stirling permutations of $S_n(i)$ for all $i$, where $n\geq 2$.
Let $\sigma'\in\mq_{n+1}$ be obtained from $\sigma\in S_n(i)$ by inserting the pair $(n+1)(n+1)$ into $\sigma$.
We distinguish the following three cases:
\begin{itemize}
  \item [\rm ($c_1$)] If $\sigma_1=\sigma_2$ and the pair $(n+1)(n+1)$ is inserted at the front of $\sigma$, then the change of labeling is illustrated as follows: $$^y\sigma_1^x\sigma_2\cdots\mapsto ^y(n+1)^x(n+1)^z\sigma_1^z\sigma_2\cdots.$$
In this case, the insertion corresponds to the rule $y\mapsto yz^2$ and $\sigma'\in S_{n+1}(i)$;
\item [\rm ($c_2$)]If $\sigma_1<\sigma_2$ and the pair $(n+1)(n+1)$ is inserted at the front of $\sigma$, then the change of labeling is illustrated as follows: $$^x\sigma_1\cdots\mapsto ^y(n+1)^x(n+1)^z\sigma_1\cdots.$$
In this case, the insertion corresponds to the rule $x\mapsto xyz$ and $\sigma'\in S_{n+1}(i+1)$;
\item [\rm ($c_3$)]If $i$ is an ascent plateau of $\sigma$, and the pair $(n+1)(n+1)$ is inserted immediately before or right after $\sigma_i$, then the change of labeling are illustrated as follows: $$\cdots\sigma_{i-1}^y\sigma_i^y\sigma_{i+1}\cdots\mapsto \cdots\sigma_{i-1}^y(n+1)^y(n+1)^z\sigma_i^z\sigma_{i+1}\cdots,$$
$$\cdots\sigma_{i-1}^y\sigma_i^y\sigma_{i+1}\cdots\mapsto \cdots\sigma_{i-1}^z\sigma_i^y(n+1)^y(n+1)^z\sigma_{i+1}\cdots.$$
In this case, the insertion corresponds to the rule $y\mapsto yz^2$ and $\sigma'\in S_{n+1}(i)$;
\item [\rm ($c_4$)]If the pair $(n+1)(n+1)$ is inserted to a position with the label $z$, then the change of labeling are illustrated as follows: $$\cdots\sigma_{i}^z\cdots\mapsto \cdots\cdots\sigma_{i}^y(n+1)^y(n+1)^z\cdots.$$
In this case, the insertion corresponds to the rule $z\mapsto y^2z$ and $\sigma'\in S_{n+1}(i+2)$.
\end{itemize}
It is routine to check that each element of $\mq_{n+1}$ can be obtained exactly once. By induction, we present a constructive proof of~\eqref{Dnx}.
Using the Leibniz's formula, we have $D^n(xy)=\sum_{k=0}^nD^k(x)D^{n-1}(y)$.
Combining~\eqref{Dnx} and Proposition~\ref{Ma13},
we get the desired formula~\eqref{Fnx}.
\end{proof}

Let $$T_n(x)=\sum_{\sigma\in\mqn}x^{\fap(\sigma)}=\sum_{k\geq1}T(n,k)x^k.$$
From the proof of~\eqref{Dnx}, we see that
the numbers $T(n,k)$ satisfy the recurrence relation
\begin{equation*}\label{Tnk-recurrence}
T(n+1,k)=kT(n,k)+T(n,k-1)+(2n-k+2)T(n,k-2).
\end{equation*}
with the initial conditions $T(0,0)=1,~T_(1,1)=1$ and $T(1,k)=0$ for $k\neq 1$.
It should be noted that $T(n,k)$ is also the number of dual Stirling permutations of order $n$ with $k$ alternating runs (see~\cite{Mawang16}).
Recall that (see~\cite[A008292]{Sloane}):
\begin{equation*}\label{Axz}
A(x,t)=\sum_{n\geq 0}A_n(x)\frac{t^n}{n!}=\frac{x-1}{x-e^{t(x-1)}}.
\end{equation*}
Hence
$$F(x,t)=\frac{x-1}{x-e^{t(x^2-1)}}.$$
Let $T(x,t)=\sum_{n\geq 0}T_n(x)\frac{t^n}{n!}$. Write the formula~\eqref{Fnx} as follows:
$$F_n(x)=\sum_{k=0}^n\binom{n}{k}T_k(x)M_{n-k}(x^2).$$
Thus,
$F(x,t)=T(x,t)M(x^2,t)$.
Combining~\eqref{Mxt}, we get
\begin{equation}\label{Txt-GF}
T(x,t)=\frac{F(x,t)}{M(x^2,t)}=\frac{x-1}{x-e^{t(x^2-1)}}\sqrt{\frac{x^2-e^{2t(x^2-1)}}{x^2-1}}.
\end{equation}
Combining~\eqref{Nxt} and~\eqref{Txt-GF}, we have $$T(x,t)N(x^2,t)=1-x+xA(x,t(1+x)).$$
Therefore, a dual formula of~\eqref{Fnx} is given as follows:
$$\sum_{\pi\in B_n}x^{\fdes(\pi)+1}=\sum_{k=0}^n\binom{n}{k}\sum_{\sigma\in\mq_{k}}x^{\fap(\sigma)}\sum_{\sigma\in\mq_{n-k}}x^{2\lap(\sigma)}.$$
for $n\geq1$.

Let $\delta_{i,j}$ be the  Kronecker delta, i.e., $\delta_{i,j}=1$ if $i=j$ and $\delta_{i,j}=0$ if $i\neq j$.
It is not hard to verify that
$T(x,t)T(-x,t)=1$. In other words, $$\sum_{k=0}^n\binom{n}{k}T_k(x)T_{n-k}(-x)=\delta_{0,n}.$$


\section{Multivariate polynomials over Stirling polynomials}\label{Section03}
\hspace*{\parindent}

Let $$C_n(x)=\sum_{\sigma\in\mqn}x^{\asc(\sigma)}.$$
The polynomials $C_n(x)$ and $N_n(x)$ respectively satisfy the following recurrence relation
$$C_{n+1}(x)=(2n+1)xC_n(x)+x(1-x)C_n'(x),$$
$$N_{n+1}(x)=(2n+1)xN_n(x)+2x(1-x)N_n'(x),$$
with the initial conditions $C_0(x)=N_0(x)=1$ (see~\cite{Bona08,Gessel78,Ma132} for instance).
In this section, we shall present a unified refinement of the polynomials $C_n(x)$ and $N_n(x)$.

In the sequel, we always assume that Stirling permutations are prepended by $0$. That is, we identify an $n$-Stirling permutation
$\sigma_1\sigma_2\cdots\sigma_{2n}$ with the word $\sigma_0\sigma_1\sigma_2\cdots\sigma_{2n}$, where $\sigma_0=0$.

\subsection{A grammatical labeling of Stirling permutations}
\hspace*{\parindent}

\begin{definition}
Let $\sigma=\sigma_1\sigma_2\cdots\sigma_{2n}\in\mqn$. For $1\leq i\leq 2n$, a double ascent of $\sigma$ is an index $i$ such that
$\sigma_{i-1}<\sigma_i<\sigma_{i+1}$, a descent-plateau of $\sigma$ is an index $i$ such that
$\sigma_{i-1}>\sigma_i=\sigma_{i+1}$.
\end{definition}
Let $\dasc(\sigma)$ and $\desp(\sigma)$ denote the numbers of double ascents and descent-plateaus
of $\sigma$, respectively. For example, $\dasc(\textbf{2}4433211\textbf{5}665)=2$ and $\desp(244\textbf{3}32\textbf{1}15665)=2$.
It is clear that
\begin{equation}\label{asc-plat}
\asc(\sigma)=\lap(\sigma)+\dasc(\sigma),
\plat(\sigma)=\lap(\sigma)+\desp(\sigma).
\end{equation}

Define
\begin{align*}
P_n(x,y,z)&=\sum_{\sigma\in\mqn}x^{\lap(\sigma)}y^{\dasc(\sigma)}z^{\desp(\sigma)}=\sum_{i,j,k}P_n(i,j,k)x^iy^jz^k,
\end{align*}
where $1\leq i\leq n,0\leq j\leq n-1,0\leq k\leq n-1$. In particular, $$P_n(x,x,1)=P_n(x,1,x)=C_n(x),~P_n(x,1,1)=N_n(x).$$
The first few of the polynomials $P_n(x,y,z)$ are given as follows:
\begin{align*}
P_1(x,y,z)&=x,\\
P_2(x,y,z)&=xy+xz+x^2,\\
P_3(x,y,z)&=x(y^2+z^2)+4x^2(y+z)+2xyz+2x^2+x^3.
\end{align*}

Now we present the second main result of this paper.
\begin{theorem}\label{mthm02}
Let $A=\{x,y,z,p,q\}$ and
\begin{equation}\label{grammar-02}
G=\{x\rightarrow xzq,y\rightarrow yzp, z\rightarrow xyz, p\rightarrow xyz, q\rightarrow xyz\}.
\end{equation}
Then
\begin{equation*}\label{Dnz}
D^n(z)=z\sum_{i,j,k}P_n(i,j,k)(xy)^iq^jp^kz^{2n-2i-j-k},
\end{equation*}
where $1\leq i\leq n,0\leq j\leq n-1,0\leq k\leq n-1$ and $2i+j+k\leq 2n$.
Set $P_n=P_n(x,y,z)$. Then the polynomials $P_n(x,y,z)$ satisfy the recurrence relation
\begin{equation}\label{Pn-recu}
P_{n+1}=(2n+1)xP_n+(xy+xz-2x^2)\frac{\partial}{\partial x}P_n+x(1-y)\frac{\partial}{\partial y}P_n+x(1-z)\frac{\partial}{\partial z}P_n,
\end{equation}
with the initial condition $P_0(x,y,z)=1$.
\end{theorem}
\begin{proof}
Now we give a labeling of $\sigma\in \mqn$ as follows:
\begin{itemize}
  \item [\rm ($L_1$)]If $i$ is a left ascent-plateau, then put a superscript label $y$ immediately before $\sigma_i$ and a superscript label $x$ right after $\sigma_i$;
 \item [\rm ($L_2$)]If $i$ is a double ascent, then put a superscript label $q$ immediately before $\sigma_i$;
\item [\rm ($L_3$)]If $i$ is a descent-plateau, then put a superscript label $p$ right after $\sigma_i$;
\item [\rm ($L_4$)]The rest positions in $\sigma$ are labeled by a superscript label $z$.
\end{itemize}
The weight of $\sigma$ is defined by $$w(\sigma)=z(xy)^{\lap(\sigma)}q^{\dasc(\sigma)}p^{\desp(\sigma)}z^{2n-2\lap(\sigma)-\dasc(\sigma)-\desp(\sigma)}.$$
For example, the labeling of $552442998813316776$ is as follows:
$$^y5^x5^z2^y4^x4^z2^y9^x9^z8^p8^z1^y3^x3^z1^q6^y7^x7^z6^z.$$

We proceed by induction on $n$.
Note that $\mq_1=\{^y1^x1^z\}$ and $$\mq_2=\{^y1^x1^y2^x2^z,~^q1^y2^x2^z1^z,~^y2^x2^z1^p1^z\}.$$
Thus the weight of $^y1^x1^z$ is given by $D(z)$ and the sum of weights of elements in $\mq_2$ is
given by $D^2(z)$, since
$D(z)=xyz$ and $D^2(x)=z(xyqz+xypz+x^2y^2)$.

Assume that the result holds for $n=m-1$, where $m\geq 3$.
Let $\sigma$ be an element counted by $P_{m-1}(i,j,k)$, and
let $\sigma'$ be an element of $\mq_{m}$ obtained by inserting the pair $mm$ into $\sigma$.
We distinguish the following five cases:
\begin{itemize}
  \item [\rm ($c_1$)]If the pair $mm$ is inserted at a position with label $x$, then the change of labeling is illustrated as follows: $$\cdots\sigma_{i-1}^y\sigma_i^x\sigma_{i+1}\cdots\mapsto \cdots\sigma_{i-1}^q\sigma_i^ym^xm^z\sigma_{i+1}\cdots.$$
In this case, the insertion corresponds to the rule $x\mapsto xzq$ and produces $i$ permutations in $\mq_m$ with $i$ left ascent-plateaus, $j+1$ double ascents and $k$ descent-plateaus;
  \item [\rm ($c_2$)]If the pair $mm$ is inserted at a position with label $y$, then the change of labeling is illustrated as follows: $$\cdots\sigma_{i-1}^y\sigma_i^x\sigma_{i+1}\cdots\mapsto \cdots\sigma_{i-1}^ym^xm^z\sigma_i^p\sigma_{i+1}\cdots.$$
In this case, the insertion corresponds to the rule $y\mapsto yzp$ and produces $i$ permutations in $\mq_m$ with $i$ left ascent-plateaus, $j$ double ascents and $k+1$ descent-plateaus;
  \item [\rm ($c_3$)]If the pair $mm$ is inserted at a position with label $z$, then the change of labeling is illustrated as follows: $$\cdots\sigma_i^z\sigma_{i+1}\cdots\mapsto \cdots\sigma_i^ym^xm^z\sigma_{i+1}\cdots.$$
In this case, the insertion corresponds to the rule $z\mapsto xyz$ and produces $2m-2-2i-j-k$ permutations in $\mq_m$ with $i+1$ left ascent-plateaus, $j$ double ascents and $k$ descent-plateaus;
  \item [\rm ($c_4$)]If the pair $mm$ is inserted at a position with label $q$, then the change of labeling is illustrated as follows: $$\cdots\sigma_i^q\sigma_{i+1}\cdots\mapsto \cdots\sigma_i^ym^xm^z\sigma_{i+1}\cdots.$$
In this case, the insertion corresponds to the rule $q\mapsto xyz$ and produces $j$ permutations in $\mq_m$ with $i+1$ left ascent-plateaus, $j-1$ double ascents and $k$ descent-plateaus;
  \item [\rm ($c_5$)]If the pair $mm$ is inserted at a position with label $p$, then the change of labeling is illustrated as follows: $$\cdots\sigma_i^p\sigma_{i+1}\cdots\mapsto \cdots\sigma_i^ym^xm^z\sigma_{i+1}\cdots.$$
In this case, the insertion corresponds to the rule $p\mapsto xyz$ and produces $k$ permutations in $\mq_m$ with $i+1$ left ascent-plateaus, $j$ double ascents and $k-1$ descent-plateaus.
\end{itemize}
By induction, we see that grammar~\eqref{grammar-02} generates all of the permutations in $\mq_m$.

Combining the above five cases, we see that
\begin{align*}
P_{n+1}(i,j,k)=&iP_n(i,j-1,k)+iP_n(i,j,k-1)+(j+1)P_n(i-1,j+1,k)+\\
                &(k+1)P_n(i-1,j,k+1)+(2n+3-2i-j-k)P_n(i-1,j,k).
\end{align*}
Multiplying both sides of the above recurrence relation by $x^iy^jz^k$ for all $i,j,k$, we get~\eqref{Pn-recu}
\end{proof}


\subsection{Equidistributed statistics}
\hspace*{\parindent}

Let $i \in [2n]$ and let $\sigma= \sigma_1\sigma_2\ldots \sigma_{2n}\in\mathcal{Q}_n.$ We define the action $\varphi_i$ as follows:
\begin{itemize}\item  If $i$ is a double ascent, then $\varphi_i(\sigma)$ is obtained by moving $\sigma_i$ to the right of the second $\sigma_i$, which forms a new pleateau $\sigma_i\sigma_i$;
\item If $i$ is a descent-plateau, then $\varphi_i(\sigma)$ is obtained by moving $\sigma_i$ to the right of $\sigma_{k}$, where
$k=\max\{j\in \{0,1,2,\ldots,i-1\}:\sigma_j < \sigma_i\}$.
\end{itemize}

For instance, if $\sigma=2447887332115665$, then $$\varphi_{1}(\sigma)=4478873322115665,~\varphi_{4}(\sigma)=2448877332115665,$$
and $\varphi_{9}(\varphi_{1}(\sigma))=\varphi_{6}(\varphi_{4}(\sigma))=\sigma$.
In recent years, the Foata and Strehl's group action has been extensively studied (see~\cite{Branden08,Lin15} for instance).
We define the Foata-Strehl action on Stirling permutations by
$$\varphi'_i(\sigma)=\left\{\begin{array}{lll}
\varphi_i(\sigma),&\text{ if $i$ is a double ascent or descent-plateau;}\\
\sigma,&\text{otherwise.}\\
\end{array}\right.$$

It is clear that the ${\varphi'_i}$'s are involutions and that they commute. Hence,
for any subset $S \subseteq[2n]$, we may define the function $\varphi'_S : \mathcal{Q}_n \mapsto \mathcal{Q}_n$ by
$\varphi'_S(\sigma)=\prod\limits_{i\in S}\varphi'_i(\sigma)$.
Hence the group $\mathbb{Z}^{2n}_2$ acts on $\mqn$ via the function $\varphi'_S$, where $S \subseteq[2n]$.

The third main result of this paper is given as follows, which is implied by~\eqref{Pn-recu}.
\begin{theorem}\label{thm-P(lap-dasc-desp)}
For any $n\geq 1$, we have
\begin{equation}\label{sym}
P_n(x,y,z)=P_n(x,z,y).
\end{equation}
Furthermore,
\begin{equation}\label{sym02}
\sum_{\sigma\in\mqn}x^{\lap(\sigma)}y^{\asc(\sigma)}=\sum_{\sigma\in\mqn}x^{\lap(\sigma)}y^{\plat(\sigma)}.
\end{equation}
\end{theorem}
\begin{proof}
For any $\sigma\in\mathcal{Q}_n$, we define
\begin{align*}
\Dasc(\sigma)&=\{i\in [2n]:\sigma_{i-1}<\sigma_i<\sigma_{i+1}\},\\
\Desp(\sigma)&=\{i\in [2n]:\sigma_{i-1}>\sigma_i=\sigma_{i+1}\},\\
\Lap(\sigma)&=\{i\in [2n]:\sigma_{i-1}<\sigma_i=\sigma_{i+1}\}.
\end{align*}
Let $S=S(\sigma)=\Dasc(\sigma)\cup \Desp(\sigma)$.
Note that $$\Dasc(\varphi'_S(\sigma))=\Desp(\sigma),~\Desp(\varphi'_S(\sigma))=\Dasc(\sigma)\text{ and }\Lap(\varphi'_S(\sigma))=\Lap(\sigma).$$
Therefore,
\begin{eqnarray*}P_n(x,y,z)&=&\sum_{\sigma\in\mqn}x^{\lap(\sigma)}y^{\dasc(\sigma)}z^{\desp(\sigma)}\\
&=&\sum_{\sigma'\in\mqn,}x^{\lap(\varphi'_{S(\sigma)}(\sigma))}y^{\desp(\varphi'_{S(\sigma)}(\sigma)))}z^{\dasc(\varphi'_{S(\sigma)}(\sigma)))}\\
&=&\sum_{\sigma\in\mqn}x^{\lap(\sigma)}z^{\dasc(\sigma)}y^{\desp(\sigma)}\\
&=&P_n(x,z,y).
\end{eqnarray*}
Combining~\eqref{asc-plat} and~\eqref{sym}, we see that
$P_n(xy,y,1)=P_n(xy,1,y)$. This completes the proof.
\end{proof}

\begin{theorem}\label{mthm03}
For $n\geq 1$, we have
\begin{equation*}\label{equi}
\sum_{\sigma\in\mqn}x^{\lap(\sigma)}y^{\dasc(\sigma)}z^{\desp(\sigma)}
=\sum_{\substack{1\leq i\leq n\\ 0\leq j \leq n-1}} \gamma_{n,i,j}x^{i}(y+z)^{j},
\end{equation*}
where  $$\gamma_{n,i,j}=\#\{\sigma\in\mqn: \lap(\sigma)=i,\dasc(\sigma)=j,\desp(\sigma)=0\}.$$
\end{theorem}
\begin{proof}
Define $$\NDP_{n,i,j}=\{\sigma\in\mqn:\lap(\sigma)=i,\dasc(\sigma)=j,\desp(\sigma)=0\}.$$
For any $\sigma\in \NDP_{n,i,j}$, let
$$[\sigma]=\{\varphi'_S(\sigma)\mid S\subseteq \Dasc(\sigma)\}.$$

For any $\sigma'\in [\sigma]$, suppose that $\sigma'=\varphi'_S(\sigma)$ for some $S\subseteq \Dasc(\sigma)$.
Then $$\lap(\sigma')=\lap(\sigma), \dasc(\sigma')=\dasc(\sigma)-|S|\text{ and }\desp(\sigma')=|S|.$$
Moreover, $\{[\sigma]\mid \sigma\in \NDP_{n,i,j}\}$ form a partition of $\mathcal{Q}_n$. Hence, \begin{eqnarray*}
&&\sum_{\sigma\in\mqn}x^{\lap(\sigma)}y^{\dasc(\sigma)}z^{\desp(\sigma)}\\
&=&\sum_{\sigma\in \NDP_n}\sum_{\sigma'\in[\sigma]}x^{\lap(\sigma')}y^{\dasc(\sigma')}z^{\desp(\sigma')}\\
&=&\sum_{\sigma\in \NDP_n}\sum_{S\subseteq \Dasc(\sigma)}x^{\lap(\varphi'_S(\sigma))}y^{\dasc(\varphi'_S(\sigma))}z^{\desp(\varphi'_S(\sigma))}\\
&=&\sum_{\sigma\in \NDP_n}\sum_{S\subseteq \Dasc(\sigma)}x^{\lap(\sigma)}y^{\dasc(\sigma)-|S|}z^{|S|}\\
&=&\sum_{\sigma\in \NDP_n}x^{\lap(\sigma)}\sum_{S\subseteq \Dasc(\sigma)}y^{\dasc(\sigma)-|S|}z^{|S|}\\
&=&\sum_{\sigma\in \NDP_n}x^{\lap(\sigma)}(y+z)^{\dasc(\sigma)}\\
&=&\sum_{i,j}\gamma_{n,i,j}x^{i}(y+z)^{j}.
\end{eqnarray*}
\end{proof}

Taking $y=z=1$ in Theorem~\ref{equi}, we have
$$N_n(x)=\sum_{\sigma\in\mqn}x^{\lap(\sigma)}=\sum_{i=1}^n\left(\sum_{j=0}^{n-1}2^j\gamma_{n,i,j}\right)x^i.$$
Let $N_n(x)=\sum_{k=1}^nN(n,k)x^k$. According to~\cite[Eq.~(24)]{Ma132},
$$N_n(x)=\sum_{k=1}^n2^{n-2k}\binom{2k}{k}k!\Stirling{n}{k}x^k(1-x)^{n-k}.$$
Thus, for $n\geq 1$, we have
$$\sum_{j=0}^{n-1}2^j\gamma_{n,i,j}=\sum_{j=1}^i(-1)^{i-j}2^{n-2j}\binom{2j}{j}\binom{n-j}{i-j}j!\Stirling{n}{j}.$$


\begin{theorem}
Let $A=\{u,v,w\}$ and
$G=\{u\rightarrow uvw,v\rightarrow 2uw, w\rightarrow uw\}$.
Then
\begin{equation}\label{Dnw}
D^n(w)=\sum_{\substack{1\leq i\leq n\\ 0\leq j \leq n-1}} \gamma_{n,i,j}u^iv^jw^{2n+1-2i-j}.
\end{equation}
Furthermore, the numbers $\gamma_{n,i,j}$ satisfy the recurrence relation
\begin{equation}\label{gamma-recu}
\gamma_{n+1,i,j}=i\gamma_{n,i,j-1}+2(j+1)\gamma_{n,i-1,j+1}+(2n+3-2i-j)\gamma_{n,i-1,j},
\end{equation}
with the initial conditions $\gamma_{1,1,0}=1$ and $\gamma_{1,i,j}=0$ for $i>1$ and $j\geq0$.
\end{theorem}
\begin{proof}
From the grammar~\eqref{grammar-02}, we see that
\begin{align*}
D(xy)&=xyz(p+q),\\
D(p+q)&=2xyz,\\
D(z)&=xyz.
\end{align*}
Set $u=xy,v=p+q$ and $w=z$. Then $D(u)=uvw,D(v)=2uw$ and $D(w)=uw$.
Combining Theorem~\ref{mthm02} and Theorem~\ref{mthm03}, we get~\eqref{Dnw}.
Since $D^{n+1}(w)=D(D^n(w))$, we obtain that
\begin{align*}
D^{n+1}(w)&=D\left(\sum_{i,j}\gamma_{n,i,j}u^iv^jw^{2n+1-2i-j}\right)\\
&=\sum_{i,j}i\gamma_{n,i,j}u^iv^{j+1}w^{2n+2-2i-j}+2\sum_{i,j}j\gamma_{n,i,j}u^{i+1}v^{j-1}w^{2n+2-2i-j}+\\
&\sum_{i,j}(2n+1-2i-j)\gamma_{n,i,j}u^{i+1}v^jw^{2n+1-2i-j}.
\end{align*}
Equating the coefficients of $u^iv^jw^{2n+1-2i-j}$ on both sides of the above equation, we obtain~\eqref{gamma-recu}.
\end{proof}

Let $G_n(x,y)=\sum_{i,j} \gamma_{n,i,j}x^iy^j$.
Multiplying both sides of the recurrence relation~\eqref{gamma-recu} by $x^iy^j$ for all $i,j$, we get
that
\begin{equation}\label{Gnxy-recu}
G_{n+1}(x,y)=(2n+1)xG_n({x,y})+(xy-2x^2)\frac{\partial}{\partial x}G_n({x,y})+(2x-xy)\frac{\partial}{\partial y}G_n({x,y}).
\end{equation}
The first few of the polynomials $G_n(x,y)$ are given as follows:
$$G_0(x,y)=1,
G_1(x,y)=x,
G_2(x,y)=xy+x^2,
G_3(x,y)=xy^2+4x^2y+2x^2+x^3.$$

\subsection{Connection with Eulerian numbers}
\hspace*{\parindent}

Recall that the {\it Eulerian numbers} are defined by $$\Eulerian{n}{k}=\#\{\pi\in\msn: \des(\pi)=k\}.$$
The numbers $\Eulerian{n}{k}$ satisfy the recurrence relation
$$\Eulerian{n+1}{k}=(k+1)\Eulerian{n}{k}+(n+1-k)\Eulerian{n}{k-1},$$
with the initial conditions $\Eulerian{1}{0}=1$ and $\Eulerian{1}{k}=0$ for $k\geq 1$.

\begin{theorem}\label{mthm05}
For $n\geq 1$ and $0\leq k\leq n-1$, we have
\begin{equation*}
\gamma_{n,n-k,k}=\Eulerian{n}{k}.
\end{equation*}
\end{theorem}
\begin{proof}
Set $a(n,k)=\gamma_{n,n-k,k}$. Then $a(n,k-1)=\gamma_{n,n-k+1,k-1}$.
Using~\eqref{gamma-recu}, it is easy to verify that
\begin{equation*}\label{nij0}
\gamma_{n,i,j}=0 {\quad\quad\text{for $i+j>n$}}.
\end{equation*}
Hence $\gamma_{n,n-k,k+1}=0$.
Therefore, the numbers $a(n,k)$ satisfy the recurrence relation
$$a(n+1,k)=(k+1)a(n,k)+(n+1-k)a(n,k-1).$$
Since the numbers $a(n,k)$ and $\Eulerian{n}{k}$ satisfy the same recurrence relation and initial conditions, so they agree.
This completes the proof.
\end{proof}

%
%
%
%
%
%
%
%

\noindent{\bf A bijective proof of Theorem~\ref{mthm05}:}
\begin{proof}
Let $\sigma\in\mqn$.
Note that every element of $[n]$ appears exactly two times in $\sigma$.
Let $\alpha(\sigma)$ be the permutation of $\msn$ obtained from $\sigma$ by deleting all of the first $i$ from left to right, where $i\in [n]$.
Then $\alpha$ is a map from $\mqn$ to $\msn$.
For example, $\alpha(\textbf{34}43\textbf{5}5\textbf{6}6\textbf{1}\textbf{2}21)=435621$.
Let $$\mathcal{D}_n=\{\sigma\in\mqn: \lap(\sigma)=i,\dasc(\sigma)=n-i,\desp(\sigma)=0\}.$$

Let $x$ be a given element of $[n]$. For any $\sigma\in\mathcal{Q}_n$, we define the action $\beta_x$ on $\mqn$ as follows:
\begin{itemize}
\item  Read $\sigma$ from left to right and let $i$ be the first index such that $\sigma_i=x$;
\item Move $\sigma_i$ to the right of $\sigma_{k}$, where
$k=\max\{j\in \{0,1,2,\ldots,i-1\}:\sigma_j < \sigma_i\}$, where $\sigma_0=0$.
\end{itemize}

For example, if $\sigma=3443578876652211$, then $$\beta_1(\sigma)=\textbf{1}344357887665221,~\beta_2(\sigma)=\textbf{2}344357887665211,~\beta_6(\sigma)=34435\textbf{6}7887652211.$$
It is clear that $\beta_x(\beta_y(\sigma))=\beta_y(\beta_x(\sigma))$ for any $x,y\in [n]$. For any $S\subseteq [n]$, let
$\beta_{S}: \mathcal{Q}_n \mapsto \mathcal{Q}_n$ be a function defined by
$$\beta_{S}(\sigma)=\prod\limits_{x\in S}\beta_x(\sigma).$$
It is easy to verify that $$\beta_{[n]}(\sigma)\in \mathcal{D}_n,~\alpha(\sigma)=\alpha(\beta_{[n]}(\sigma)), \beta_{[n]}(\sigma)=\sigma~{\text {if $\sigma\in\mathcal{D}_n$}}.$$

Let $\alpha|_{\mathcal{D}_n}$ denote the restriction of the map $\alpha$ on the set $\mathcal{D}_n$. Then $\alpha|_{\mathcal{D}_n}$ is a map from
${\mathcal{D}_n}$ to $\msn$.
Let $\pi=\pi(1)\pi(2)\cdots\pi(n)\in\msn$. The inverse ${\alpha|_{\mathcal{D}_n}^{-1}}$ is defined as follows:
\begin{itemize}
\item  let $\sigma=\sigma_1\sigma_2\ldots \sigma_{2n}$ be the Stirling permutation such that $\sigma_{2i-1}=\sigma_{2i}=\pi(i)$ for each $i=1,2,\ldots,n$;
\item let $S(\pi)=\{\pi_i: \pi_{i-1}>\pi_i, 2\leq i\leq n\}$;
\item let ${\alpha|_{\mathcal{D}_n}^{-1}}(\pi)=\beta_{S(\pi)}(\sigma)$.
\end{itemize}
 Note that $$\lap({\alpha|_{\mathcal{D}_n}^{-1}}(\pi))+\dasc({\alpha|_{\mathcal{D}_n}^{-1}}(\pi))=n\text{ and }\dasc({\alpha|_{\mathcal{D}_n}^{-1}}(\pi))=\des(\pi).$$
Then $\alpha|_{\mathcal{D}_n}$ is a bijection from
${\mathcal{D}_n}$ to $\msn$. This completes the proof.
\end{proof}

\begin{example}
The bijection between $\ms_3$ and ${\mathcal{D}_3}$ is demonstrated as follows:
\begin{align*}
123&\leftrightarrow 112233 ~(S=\emptyset)\leftrightarrow \beta_S(112233)=112233;\\
132&\leftrightarrow 113322 ~(S=\{2\})\leftrightarrow  \beta_S(113322)=112332;\\
213&\leftrightarrow 221133 ~(S=\{1\})\leftrightarrow  \beta_S(221133)=122133;\\
231&\leftrightarrow 223311 ~(S=\{1\})\leftrightarrow  \beta_S(223311)=122331;\\
312&\leftrightarrow 331122 ~(S=\{1\})\leftrightarrow  \beta_S(331122)=133122;\\
321&\leftrightarrow 332211 ~(S=\{1,2\})\leftrightarrow  \beta_{S}(332211)=123321.
\end{align*}
\end{example}

\section{Concluding remarks}\label{Section-4}
\hspace*{\parindent}

In this paper, we introduce several variants of the ascent-plateau statistic on Stirling permutations.
Recall that Park~\cite{Park94} studied the $(p,q)$-analogue of the descent polynomials of Stirling permutations:
\begin{align*}
C_n(x,p,q)=\sum_{\sigma\in\mqn}x^{\des(\sigma)}p^{\inv(\sigma)}q^{\maj(\sigma)}.
\end{align*}
It would be interesting to study the relationship between $C_n(x,p,q)$ and the following polynomials:
\begin{align*}
\sum_{\sigma\in\mqn}x^{\ap(\sigma)}y^{\lap(\sigma)}p^{\inv(\sigma)}q^{\maj(\sigma)}.
\end{align*}
In~\cite{Egge10}, Egge introduced the definition of Legendre-Stirling permutation,
which shares similar properties with Stirling permutation.
One may study the ascent-plateau statistic on Legendre-Stirling permutations.


\end{document}